\documentclass[amssymb, 11pt]{amsart}
\usepackage{latexsym}
\usepackage{young}

\newlength{\standardunitlength}
\newcommand{\bea}{\begin{eqnarray}}
\newcommand{\ena}{\end{eqnarray}}
\newcommand{\beas}{\begin{eqnarray*}}
\newcommand{\enas}{\end{eqnarray*}}

\newcommand{\ignore}[1]{}

\newtheorem{prop}{Proposition}[section]

\newtheorem{lemma}[prop]{Lemma}
\newtheorem{cor}[prop]{Corollary}
\newtheorem{theorem}[prop]{Theorem}

\begin{document}

\title [Derangements in finite classical groups] {Derangements in finite classical groups and characteristic polynomials of random matrices}

\author{Jason Fulman}
\address{Department of Mathematics\\
        University of Southern California\\
        Los Angeles, CA, 90089-2532, USA}
\email{fulman@usc.edu}

\author{Robert Guralnick}
\address{Department of Mathematics\\University of Southern California \\ Los Angeles, CA, 90089-2532, USA}
\email{guralnic@usc.edu}

\date{Submitted August 8, 2025; Revised January 27, 2026} 

\thanks{Fulman was partially supported by Simons Foundation grants 400528 and 917224.  Guralnick was partially supported by NSF grant DMS 1901595 and Simons Foundation Fellowship 609771. We thank the referee for a careful reading of the manuscript and helpful remarks. }

\begin{abstract}   We first obtain explicit upper bounds for the proportion of elements in a finite classical group $G$ with a given characteristic polynomial.   We use this to complete the proof that the proportion of elements of a finite classical group $G$ which lie in a proper irreducible
subgroup tends to $0$ as  the dimension of the natural module goes to $\infty$. 
This result is analogous to the result of Luczak and Pyber \cite{LP} that the proportion of elements of the symmetric group $S_n$
which are contained in a proper transitive subgroup other than the alternating
group goes to $0$ as $n \rightarrow \infty$.   We also show that the probability that $3$ random elements of $SL(n,q)$ invariably
generate goes to $0$ as $n \rightarrow \infty$.     
\end{abstract}

\maketitle

\section{Introduction}

Let $G$ be a classical group over a finite field.    We first study the number of elements
of $g$ with a fixed characteristic polynomial and in particular obtain explicit upper bounds. For example, we prove the following result:

\begin{theorem} \label{t:charpoly} For any polynomial $\phi$, 
the chance that a random element of $GL(n,q)$ has
characteristic polynomial $\phi$ is at most \[ \frac{e^6 [1+\log_q(n+1)]^6}{q^n} .\]
\end{theorem} 

We prove similar results for all the classical groups.   See Section \ref{charpoly}.   One can obtain
exact formulas for these numbers (and some were known) but this does not always lead to easily computed bounds which are more useful. 
In addition to the applications in this paper, our results might be useful in the context of \cite{GR}. 
Our results are different from but consistent with the findings of \cite{HS}, who showed that except for small degree factors, the characteristic polynomial of a random element of $GL(n,q)$ is like a random polynomial. 

We had obtained similar results for conjugacy classes in \cite{FG1} and the bounds there have proved to be tremendously useful.  Of course,
the upper bound on characteristic polynomials also gives a bound on conjugacy class sizes (since conjugate elements have the same characteristic polynomial)
but the bounds in \cite{FG1} are better.  

Section \ref{extension} studies the proportion of derangements for a finite classical group acting on the cosets
of an extension field subgroup.  
For instance it shows that for any prime $b$ dividing $n$, the proportion of elements of $GL(n,q)$
contained in a conjugate of $GL(n/b,q^b).b$ tends to $0$ as $n \rightarrow \infty$. The $.b$
notation means semidirect product with the cyclic group of order $b$ generated by the map $x \mapsto x^q$
on $\mathbb{F}_{q^b}^*$, the non-zero elements of a finite field of size $q^b$.   Thus the proportion of derangements goes to 
$1$ as $n \rightarrow \infty$ and indeed the same is true for the proportion of elements which are derangements on all extension
field subgroups.  

Combining the results of Section \ref{extension} with our results in \cite{FG1, FG2} (which treated all the other classes of irreducible subgroups)
 completes the proof that for a finite classical group $G$ of rank $r$, the proportion
of elements lying in a proper irreducible subgroup tends to $0$ as $r \rightarrow \infty$. This was proved by
Shalev \cite{Sh} for $G=GL(n,q)$ almost 30 years ago (for $q$ fixed).  Shalev used difficult results of Schmutz \cite{Sc} on the order
of random elements of $GL(n,q)$. Unfortunately Schmutz's work seems very difficult to extend to the other
finite classical groups. A version of Shalev's result for finite classical groups was proved for $q  \rightarrow \infty$ in \cite{FG2}.  If the dimension
of the natural module is fixed,  this fails. Any subgroup containing a maximal torus will contain
a proportion of elements of $G$ bounded away from $0$ for fixed $n$. 

 Our approach has two parts.
The first part is to show (in Section \ref{charpoly}) that not too many elements of a finite classical group can
have a given characteristic polynomial. The second part is to prove (in Section \ref{extension}) that the possible
characteristic polynomials of elements in maximal irreducible subgroups are quite restricted. Our original approach
to bounding the number of characteristic polynomials with certain restrictions in Section \ref{extension} used
generating functions. Daniele Garzoni pointed out to us that some of these lemmas were known by other methods
\cite{EG}, so we have removed our generating function approach except for the case of $GL(n,q)$. For $GL(n,q)$,
Lemma \ref{boundpolyGL} and the corresponding result from \cite{EG} are exactly the same, and we thought the
reader might enjoy seeing our generating function approach. We also note that the methods of our paper could be used
to give an alternate proof of some of the results in \cite{FG1, FG2}.  

This gives the following result:

\begin{theorem} \label{t:main2}  Let $G$ be a finite classical group with a natural module of dimension $n$.
Let $X$ be the set of all maximal subgroups of $G$ which act irreducibly on the natural module and
which do not contain the derived subgroup of $G$.  If $G$ is a symplectic group in characteristic $2$,
we exclude the orthogonal subgroups. Then as $n \rightarrow \infty$, 
$| \cup_{H \in X} H |/|G| \rightarrow 0$.
\end{theorem} 

By classical group, we generally mean a linear, unitary, symplectic or orthogonal group.  More precisely, we are proving
the above results for $GL(n,q)$, $GU(n,q)$, $Sp(2n,q)$ and $O^{\pm}(n,q)$.  Note that the results for $GL(n,q)$ and $GU(n,q)$ imply
the result for $SL(n,q)$ and $SU(n,q)$ when $q$ is bounded since then the index is at most $q$.   Similarly, the result for the orthogonal
group implies the result for the special orthogonal group since the index is $2$. 
For $q \rightarrow \infty$, the result was already proved in \cite{FG2}.

Recall that elements $g_1, \ldots, g_r$ are said to invariably generate $G$ if for any $h_i$ conjugate to $g_i$, 
we have that $G = \langle h_1, \ldots, h_r \rangle$.   Note a consequence of the previous result is that if $G$ is a finite classical
simple group of dimension $n$, the probability that $r$ random elements of $G$ invariably generate $G$ is asymptotically
close to the probability that $r$ random elements generate an irreducible subgroup.

In a follow-up  preprint \cite{FG3}, we gave another application of these results. Let $P_x(G)$ be the probability that
$x$ and a random element of $G$ generate $G$. We show that  there is an explicit universal
constant $\delta > 0$ so that for $n$ sufficiently large and $x$ any non-identity element of $G$, $P_x(G) > \delta$
(if $q \rightarrow \infty$, then $P_x(G) \rightarrow 1$ \cite{GLSS}).   This fails for alternating groups. 
In \cite{FGG}, we use different methods to prove an extension of this
to the case of almost simple groups (and also for elements in a given coset of the simple group).

In the penultimate section, we give an elementary discussion of the case of bounded rank and the field size increasing
and also work out the lower bound in the case of type A groups.  In the final section we generalize a result of McKemmie \cite{Mc}
for the special case of $SL(n,q)$.  In some sense this is unrelated to the previous result.  However, the proofs both rely on counting
unipotent elements commuting with fixed a semisimple element.    

\begin{theorem}  The probability that $3$ random elements of $SL(n,q)$ invariably generate an irreducible subgroup goes to $0$ as $n \rightarrow \infty$
(independently of $q$).
\end{theorem}

\section{Characteristic polynomials} \label{charpoly}

This section studies the proportion of elements of a finite classical group $G$ with a given characteristic polynomial.
While exact formulas were previously known in most cases, our applications require upper bounds, which are
not immediately obvious from exact formulas. The linear, unitary, symplectic, and orthogonal groups are considered
in Subsections \ref{subGL}, \ref{subU}, \ref{subSp}, \ref{subO} respectively.

\subsection{$GL(n,q)$} \label{subGL}

The main point of this subsection is to prove the following theorem. In its proof, and throughout the paper,
we let $N(q;d)$ denote the number of monic degree $d$ irreducible polynomials over $\mathbb{F}_q$ with
non-zero constant term.

\begin{theorem} \label{GLbound} For any polynomial $\phi$, the chance that a random element of $GL(n,q)$ has
characteristic polynomial $\phi$ is at most
\[ \frac{e^6 [1+\log_q(n+1)]^6}{q^n} .\]
\end{theorem}

\begin{proof} Let $\phi$ be a monic degree $n$ polynomial over $\mathbb{F}_q$ with non-zero constant term.
Suppose $\phi$ factors into a product of irreducibles as $\phi = \prod \phi_i^{j_i}$ where
$\phi_i$ is irreducible of degree $m(\phi_i)$. Then from \cite{R}, the
chance that a random element of $GL(n,q)$ has characteristic polynomial $\phi$ is
\[ \prod_{i \geq 1} \frac{q^{m(\phi_i) j_i(j_i-1)}}{|GL(j_i,q^{m(\phi_i)})|} .\]
This is equal to \begin{equation} \label{starred}
 \frac{1}{q^n} \prod_{i \geq 1} \frac{1}{(1-1/q^{m(\phi_i)})(1-1/q^{2m(\phi_i)})
\cdots (1-1/q^{j_im(\phi_i)})} \end{equation}

It is easily proved that for $q \geq 2$ and non-negative integers $a$,
\[ \frac{1}{(1-1/q^a)(1-1/q^{2a}) \cdots} \leq \frac{1}{1-1/q^a-1/q^{2a}} \leq (1+6/q^a) \leq e^{6/q^a} .\]
The first inequality followed from Euler's pentagonal number theorem. We remark that the $6$ can be replaced
by $3$ when $q=3$ and by $2$ when $q \geq 4$, but we do not need this. Thus the proportion of elements of
$GL(n,q)$ with characteristic $\phi$ is at most
\[ \frac{1}{q^n} \prod_i e^{6/q^{m(\phi_i)}} .\] The worst case of this bound is when the characteristic
polynomial uses as many small degree factors as possible, each with multiplicity one. Hence if $r$ is such
that $\sum_{d=1}^r d N(q;d) \geq n$, this bound is at most
\[ \frac{1}{q^n} \prod_{d=1}^r e^{6 N(q;d)/q^d} .\] Since $N(q;d) \leq q^d/d$, this bound is at most
\[ \frac{1}{q^n} e^{\sum_{d=1}^r 6/d} \leq \frac{e^{6(1+\log(r))}}{q^n} = \frac{e^6 r^6}{q^n} .\]

Observe that $r$ can be taken to be the smallest integer such that $q^r-1 \geq n$, since
\[ \sum_{d=1}^r d N(q;d) \geq \sum_{d|r} dN(q;d) = q^r-1 .\] Hence we can take
$r=1 + \log_q(n+1)$, which completes the proof.
\end{proof}

{\it Remark:}  As noted in the proof, the $e^6$ can be replaced by $e^3$ for $q = 3$ and 
$e^2$ for $q \ge 4$.  \\

{\it Remark:} It is obvious from \eqref{starred} that a lower bound on the proportion
of elements of $GL(n,q)$ with a given characteristic polynomial is $\frac{1}{q^n-1}$.  \\

{\it Remark:} The bound from Theorem \ref{GLbound} also holds for random $n \times n$ matrices (not necessarily invertible) over
the field of $q$ elements. This follows without difficulty from the fact that the proportion of such matrices which have characteristic polynomial $\phi$
is obtained by multiplying \eqref{starred} by $(1-1/q)(1-1/q^2) \cdots (1-1/q^n)$. We will consider this problem for the other Lie algebras in a sequel.  

\subsection{$GU(n,q)$} \label{subU}

In this section we focus on $GU(n,q)$. Given a polynomial $\phi \in \mathbb{F}_{q^2}[t]$ with non-zero constant term,
\[ \phi(t)=t^n+a_{n-1}t^{n-1} + \cdots + a_1t + a_0 ,\] we define the conjugate $\tilde{\phi}$ by
\[ \tilde{\phi}(t) = t^n + (a_1a_0^{-1})^q t^{n-1} + \cdots + (a_{n-1}a_0^{-1})^q t + (a_0^{-1})^q .\]
We let $\tilde{N}(q;d)$ denote the number of monic irreducible self-conjugate polynomials of degree $d$
over $\mathbb{F}_{q^2}$. We also let $\tilde{M}(q;d)$ denote the number of (unordered) conjugate pairs $\{ \phi,\tilde{\phi} \}$
of monic irreducible polynomials of degree $d$ over $\mathbb{F}_{q^2}$ that are not self-conjugate. As in the
previous section, we let $m(\phi)$ denote the degree of a polynomial $\phi$.

The following lemma is crucial.

\begin{lemma} \label{uform} (\cite{F}) Let $\phi$ be a monic, self-conjugate polynomial of degree $n$ which factors
into irreducibles as \[ \phi = \prod_i \phi_i^{j_i} \prod_{i'} \left[ \phi_{i'} \tilde{\phi_{i'}} \right]^{j_{i'}}, \]
where $\phi_i=\tilde{\phi_i}$ and $\phi_{i'} \neq \tilde{\phi_{i'}}$. Then the proportion of elements of
$GU(n,q)$ with characteristic polynomial $\phi$ is:

\[ \prod_i \frac{q^{m(\phi_i)j_i(j_i-1)}}{|GU(j_i,q^{m(\phi_i)})|} \prod_{i'} \frac{q^{2m(\phi_{i'})j_{i'}(j_{i'}-1)}}
{|GL(j_{i'},q^{2m(\phi_i')})|} .\]
\end{lemma}

Theorem \ref{ubound} is the main result of this subsection.

\begin{theorem} \label{ubound} For any polynomial $\phi$, the chance that a random element of $GU(n,q)$ has
characteristic polynomial $\phi$ is at most
\[ \frac{e [2+\log_q(n+1)]}{q^n} .\]
\end{theorem}

\begin{proof} The theorem is immediate for $n=1$, so suppose that $n>1$.

 From Lemma \ref{uform}, the proportion of elements of $GU(n,q)$ with characteristic polynomial $\phi$
is equal to
\begin{eqnarray*}
& & \frac{1}{q^n} \prod_i \frac{1}{(1+1/q^{m(\phi_i)}) (1-1/q^{2m(\phi_i)}) \cdots (1- (-1)^{j_i}/q^{j_i m(\phi_i)})} \\
& & \cdot \prod_{i'} \frac{1}{(1-1/q^{2m(\phi_{i'})}) (1-1/q^{4m(\phi_{i'})}) \cdots (1-1/q^{2 j_{i'} m(\phi_{i'})})}.
\end{eqnarray*}
The $i$ terms in the above product are all less than 1, so the proportion of elements of $GU(n,q)$ with characteristic
polynomial $\phi$ is at most
\[ \frac{1}{q^n} \prod_{i'} \frac{1}{(1-1/q^{2m(\phi_{i'})}) (1-1/q^{4m(\phi_{i'})}) \cdots (1-1/q^{2 j_{i'} m(\phi_{i'})})} .\]

Now if $q \geq 2$ and $a \geq 2$, one has that
\[ \frac{1}{(1-1/q^a)(1-1/q^{2a}) \cdots} \leq \frac{1}{1-1/q^a-1/q^{2a}} \leq 1 + \frac{2}{q^a} \leq e^{2/q^a} .\]
(The first inequality is from Euler's pentagonal number theorem). Thus the proportion of elements of $GU(n,q)$ with
characteristic polynomial $\phi$ is at most
\[ \frac{1}{q^n} \prod_{i'} e^{2/q^{2m(\phi_{i'})}}.\] The worst case of this bound is when the characteristic polynomial
uses as many small degree conjugate pairs of non-self-conjugate polynomials as possible, each with multiplicity one. Hence if $r$
is such that $\sum_{d=1}^r 2d \tilde{M}(q;d) \geq n$, this bound is at most
\[ \frac{1}{q^n} \prod_{d=1}^r e^{2 \tilde{M}(q;d)/q^{2d}} .\]
Since $\tilde{M}(q;d) \leq q^{2d}/(2d)$, this bound is at most
\[ \frac{1}{q^n} e^{\sum_{d=1}^r 1/d} \leq \frac{e^{(1+\log(r))}}{q^n} = \frac{er}{q^n} .\]

From the proof of Theorem 6.7 of \cite{FG1}, if $r$ is odd and $q^r \geq n+1$, one has that $\sum_{d=1}^r 2d \tilde{M}(q;d) \geq n$
for $n>1$. Hence one can take $r \leq 2+ \log_q(n+1)$, and the final bound of the theorem becomes
\[ \frac{e [2+\log_q(n+1)]}{q^n} .\]
\end{proof}

\subsection{$Sp(2n,q)$} \label{subSp}

In this subsection we focus on $Sp(2n,q)$. Given a polynomial $\phi \in \mathbb{F}_{q}[t]$ with non-zero constant term
\[ \phi(t) = t^n+a_{n-1}t^{n-1} + \cdots + a_1t + a_0,\]
we define the conjugate polynomial \[ \phi^*(t)=t^n+a_1a_0^{-1}t^{n-1} + \cdots + a_{n-1}a_0^{-1}t + a_0^{-1} .\]
We also let $M^*(q;d)$ denote the number of (unordered) conjugate pairs $\{ \phi,\phi^* \}$ of monic,
irreducible non-self-conjugate polynomials of degree $d$ over $\mathbb{F}_q$. As in the
previous sections, we let $m(\phi)$ denote the degree of a polynomial $\phi$.

The following lemma from \cite{F} is crucial.

\begin{lemma} \label{countSp}
\begin{enumerate}

\item Suppose that the characteristic is odd. Let $\phi$ be a monic, self-conjugate polynomial of degree $2n$
which factors into irreducibles as
\[ \phi = (z-1)^{2a} (z+1)^{2b} \prod_i \phi_i^{j_i} \prod_{i'} \left[ \phi_{i'} \phi_{i'}^* \right]^{j_{i'}}, \]
where $\phi_i=\phi_i^*$ and $\phi_{i'} \neq \phi_{i'}^*$. Then the proportion of elements of $Sp(2n,q)$ with
characteristic polynomial $\phi$ is:

\[ \frac{q^{2a^2}}{|Sp(2a,q)|} \frac{q^{2b^2}}{|Sp(2b,q)|} \prod_i \frac{q^{m(\phi_i)j_i(j_i-1)/2}}{|GU(j_i,q^{m(\phi_i)/2})|}
\prod_{i'} \frac{q^{m(\phi_{i'})j_{i'}(j_{i'}-1)}}{|GL(j_{i'},q^{m(\phi_{i'})})|}. \]

\item Suppose that the characteristic is even. Let $\phi$ be a monic, self-conjugate polynomial of degree $2n$
which factors into irreducibles as
\[ \phi = (z-1)^{2a} \prod_i \phi_i^{j_i} \prod_{i'} \left[ \phi_{i'} \phi_{i'}^* \right]^{j_{i'}}, \]
where $\phi_i=\phi_i^*$ and $\phi_{i'} \neq \phi_{i'}^*$. Then the proportion of elements of $Sp(2n,q)$ with
characteristic polynomial $\phi$ is:

\[ \frac{q^{2a^2}}{|Sp(2a,q)|} \prod_i \frac{q^{m(\phi_i)j_i(j_i-1)/2}}{|GU(j_i,q^{m(\phi_i)/2})|}
\prod_{i'} \frac{q^{m(\phi_{i'})j_{i'}(j_{i'}-1)}}{|GL(j_{i'},q^{m(\phi_{i'})})|}. \]

\end{enumerate}
\end{lemma}

Theorem \ref{sym} is the main result of this subsection.

\begin{theorem} \label{sym}

\begin{enumerate}
\item Suppose that the characteristic is odd. Then for any polynomial $\phi$, the chance that a random element
of $Sp(2n,q)$ has characteristic polynomial $\phi$ is at most
\[ \frac{20[\log_q(4n)+4]^{3/2}}{q^n}. \]

\item Suppose that the characteristic is even. Then for any polynomial $\phi$, the chance that a random element
of $Sp(2n,q)$ has characteristic polynomial $\phi$ is at most
\[ \frac{242 [\log_q(4n)+4]^3}{q^n}. \]
\end{enumerate}
\end{theorem}

\begin{proof} First suppose that the characteristic is odd. From Lemma \ref{countSp}, the proportion of
elements of $Sp(2n,q)$ with characteristic polynomial $\phi$ is equal to
\begin{eqnarray*}
& & \frac{1}{q^n} \frac{1}{(1-1/q^2) \cdots (1-1/q^{2a})} \frac{1}{(1-1/q^2)\cdots (1-1/q^{2b})} \\
& & \cdot \prod_i \frac{1}{(1+1/q^{m(\phi_i)/2})(1-1/q^{2m(\phi_i)/2}) \cdots (1-(-1)^{j_i}/q^{j_im(\phi_i)/2})} \\
& & \cdot \prod_{i'} \frac{1}{(1-1/q^{m(\phi_{i'})}) (1-1/q^{2m(\phi_{i'})}) \cdots (1-1/q^{j_{i'}m(\phi_{i'})})}.
\end{eqnarray*} Since $q \geq 3$,
\[ \frac{1}{(1-1/q^2)(1-1/q^4) \cdots (1-1/q^{2a})} \leq 1.2 \] and
\[ \frac{1}{(1-1/q^2)(1-1/q^4) \cdots (1-1/q^{2b})} \leq 1.2. \] Combining this with the observation that the $i$
terms in the above product are all less than $1$ shows that the proportion of elements of $Sp(2n,q)$ with
characteristic polynomial $\phi$ is at most
\[ \frac{1.5}{q^n} \prod_{i'} \frac{1}{(1-1/q^{m(\phi_{i'})}) (1-1/q^{2m(\phi_{i'})}) \cdots (1-1/q^{j_{i'}m(\phi_{i'})})}.\]

Now if $q \geq 3$ and $a \geq 1$, one has that
\[ \frac{1}{(1-1/q^a)(1-1/q^{2a}) \cdots} \leq \frac{1}{1-1/q^a-1/q^{2a}} \leq 1 + \frac{3}{q^a} \leq e^{3/q^a}. \]
(The first inequality is from Euler's pentagonal number theorem). Thus the proportion of elements of $Sp(2n,q)$ with
characteristic polynomial $\phi$ is at most \[ \frac{1.5}{q^n} \prod_{i'}e^{3/q^{m(\phi_{i'})}}.\] The worst
case of this bound is when the characteristic polynomial uses as many small degree conjugate pairs of non-self-conjugate
polynomials as possible, each with multiplicity one. Hence if $r$ is such that $\sum_{d=1}^r 2d M^*(q;d) \geq 2n$,
this bound is at most \[ \frac{1.5}{q^n} \prod_{d=1}^r e^{3M^*(q;d)/q^d} .\] Since $M^*(q;d) \leq q^d/(2d)$, this
bound is at most \[ \frac{1.5}{q^n} e^{\frac{3}{2} \sum_{d=1}^r 1/d} \leq \frac{1.5}{q^n} e^{\frac{3}{2}(1+\log(r))} =
\frac{1.5 (er)^{3/2}}{q^n}. \] From the proof of Theorem 6.13 in \cite{FG1}, one can take
\[ r \leq 2[\log_q(4n)+4], \] yielding the bound of part one of the theorem.

Next suppose that the characteristic is even. Then arguing as in part one shows that the proportion of elements of $Sp(2n,q)$ with
characteristic polynomial $\phi$ is at most
\[ \frac{1.5}{q^n} \prod_{i'} \frac{1}{(1-1/q^{m(\phi_{i'})}) (1-1/q^{2m(\phi_{i'})}) \cdots (1-1/q^{j_{i'}m(\phi_{i'})})}.\]

Now if $q \geq 2$ and $a \geq 1$, one has that \[ \frac{1}{(1-1/q^a)(1-1/q^{2a}) \cdots} \leq \frac{1}{1-1/q^a-1/q^{2a}}
 \leq 1 + \frac{6}{q^a} \leq e^{6/q^a}. \] Thus the proportion of elements of $Sp(2n,q)$ with
characteristic polynomial $\phi$ is at most \[ \frac{1.5}{q^n} \prod_{i'}e^{6/q^{m(\phi_{i'})}}.\] Again the worst
case of this bound is when the characteristic polynomial uses as many small degree conjugate pairs of non-self-conjugate
polynomials as possible, each with multiplicity one. If $r$ is such that $\sum_{d=1}^r 2d M^*(q;d) \geq 2n$,
this bound is at most \[ \frac{1.5}{q^n} \prod_{d=1}^r e^{6M^*(q;d)/q^d} .\] Since $M^*(q;d) \leq q^d/(2d)$, this
bound is at most \[ \frac{1.5}{q^n} e^{3\sum_{d=1}^r 1/d} \leq \frac{1.5}{q^n} e^{3(1+\log(r))} =
\frac{1.5 (er)^3}{q^n}. \] As in the first part, one can take \[ r \leq 2[\log_q(4n)+4], \] yielding the bound of part two of the theorem.
\end{proof}

\subsection{$O^{\pm}(n,q)$} \label{subO}

Finally, we consider the orthogonal groups, beginning with odd characteristic. Lemma \ref{countO} is from \cite{F} and
uses the same notation as we used for the symplectic groups.

\begin{lemma} \label{countO} Suppose that the characteristic is odd. Suppose that $\phi$ is a monic, self-conjugate polynomial of
degree $n$ which factors into irreducibles as
\[ \phi = (z-1)^a (z+1)^b \prod_i \phi_i^{j_i} \prod_{i'} [\phi_{i'} \phi_{i'}^*]^{j_{i'}} .\] Then the sum of the
proportions of elements in $O^+(n,q)$ and $O^-(n,q)$ with characteristic polynomial $\phi$ is
\[ F(a) F(b) \prod_i \frac{q^{m(\phi_i)j_i(j_i-1)/2}}{|GU(j_i,q^{m(\phi_i)/2})|}
\prod_{i'} \frac{q^{m(\phi_{i'})j_{i'}(j_{i'}-1)}}{|GL(j_{i'},q^{m(\phi_{i'})})|}, \] where

\[ F(m) =             \left\{ \begin{array}{ll}
\frac{q^{m^2/2}}{|Sp(m,q)|} & \mbox{if $m$ is even}\\
\frac{q^{(m-1)^2/2}}{|Sp(m-1,q)|} & \mbox{if $m$ is odd}\end{array}
                        \right.  \]
\end{lemma}

Arguing as for the symplectic case yields the following result.

\begin{theorem} \label{ofirst} Suppose that the characteristic is odd. Then for any polynomial $\phi$, the chance that a random
element of $O^{\pm}(n,q)$ has characteristic polynomial $\phi$ is at most
\[ \frac{27 [\log_q(4r)+4]^{3/2}}{q^{r-1}} \] where $r=n/2$ if $n$ is even and $r=(n-1)/2$. 
\end{theorem}

We remark that the extra factor of $q$ in the numerator (as compared to the corresponding theorem for the symplectic group) only
arises when both $a$ and $b$ are odd in Lemma \ref{countO}.      Note that this  only occurs
for $n$ even and the elements have determinant $-1$.  

For orthogonal groups in even characteristic, it is not necessary to treat the odd dimensional groups, as these are isomorphic
to symplectic groups, and the isomorphism multiplies the characteristic polynomial by a factor of $z-1$. For even dimensional
orthogonal groups, one obtains the following theorem.

\begin{theorem} \label{ofirst2} Suppose that the characteristic is even. Then for any polynomial $\phi$, the chance that a random
element of $O^{\pm}(2n,q)$ has characteristic polynomial $\phi$ is at most
\[ \frac{352 [\log_q(4n)+4]^3}{q^{n-1}}.\]
\end{theorem}

\begin{proof} The argument is very similar to that of the symplectic case, the main difference being that Steinberg's formula
for the number of unipotent elements is not applicable as $O^{\pm}(2n,q)$ is disconnected.  Fortunately, the number of unipotent elements in even characteristic orthogonal groups $O^{\pm}(2n,q)$ was computed
 in \cite{FG1} to be $q^{2n^2-2n+1} \left( 1+\frac{1}{q} \mp  \frac{1}{q^n} \right)$. 
 \end{proof}

We remark that the extra factor of $q$ in the numerator (as compared to the corresponding theorem for the symplectic group) really
is needed, as can be seen by considering the case where $\phi=(z-1)^{2n}$, that is the case of unipotent elements.

If $1$ is not a root of $\phi$,  then the centralizer is contained in $SO$ and the argument shows that we only need a $q^n$ in the denominator.

\section{Derangements for extension field subgroups} \label{extension}

This section studies derangements for extension field subgroups. The linear, unitary, symplectic,
and orthogonal groups are considered in Subsections \ref{extGL}, \ref{extU}, \ref{extSp},
\ref{extO} respectively. The remaining other cases are considered in Subsection \ref{extother}.

We only need consider the case that the degree of the extension is prime (otherwise the subgroup is not maximal).

\subsection{$GL(n,q)$} \label{extGL}

This subsection treats the case of $GL(n/b,q^b).b$ inside of $GL(n,q)$. We treated this by a different technique in the
appendix to \cite{FG2}. However the proof technique used there does not easily extend to the other classical groups.

Recall from \cite{FG2} that if $b$ is prime and an element of $GL(n,q)$ is contained in a conjugate of $GL(n/b,q^b)$, then every
factor of its characteristic polynomial either has degree a multiple of $b$ or has every Jordan block size occur with
multiplicity a multiple of $b$. In particular, every factor of its characteristic polynomial either has degree
a multiple of $b$ or occurs with multiplicity a multiple of $b$. Lemma \ref{boundpolyGL} bounds the number
of characteristic polynomials of $GL(n,q)$ satisfying this condition.

As in Section \ref{charpoly}, we let $N(q;d)$ denote the number of
monic degree $d$ irreducible polynomials over $\mathbb{F}_q$ with non-zero constant term. 

The following lemma is known; see part 1 of Proposition 2.7 of \cite{EG}. However our generating function approach
is different.

\begin{lemma} \label{boundpolyGL} The number of characteristic polynomials of $GL(n,q)$ such that every
irreducible factor either has degree a multiple of $b$ or occurs with multiplicity a multiple of $b$ is at
most $B q^n/n^{1/2}$ for a universal constant $B$.
\end{lemma}

\begin{proof} Throughout this proof we use the notation that for $f(u)=\sum_{n \geq 0} f_n u^n$ and
$g(u)=\sum_{n \geq 0} g_n u^n$,  $f  \ll g$ means that $f_n \leq g_n$ for all $n$.

The number of polynomials satisfying the condition of the lemma is at most the coefficient
of $u^n$ in \[ \prod_d (1-u^{bd})^{-N(q;bd)} \prod_d (1-u^{bd})^{-N(q;d)} .\] Indeed, this upper bound
is an equality if $d$ in the second product ranges over all positive integers which are not a multiple of $b$.

Using the well-known fact (see for instance Lemma 1.3.10 of \cite{FNP}) that \[ \prod_d (1-u^d)^{-N(q;d)} = \frac{1-u}{1-uq} ,\] one
 concludes that \[ \prod_d (1-u^{bd})^{-N(q;d)} = \frac{1-u^b}{1-u^bq} \ll \frac{1}{1-u^bq} .\]

From the appendix of \cite{FG2}, one has the inequality $N(q;bd) \leq \frac{1}{b} N(q^b;d)$. Thus
\begin{eqnarray*}
 \prod_d (1-u^{bd})^{-N(q;bd)} & \ll & \left[ \prod_d (1-u^{bd})^{-N(q^b;d)} \right]^{1/b} \\
 & = & \left( \frac{1-u^b}{1-u^bq^b} \right)^{1/b} \\
 & \ll & \left( \frac{1}{1-u^bq^b} \right)^{1/b}. \end{eqnarray*}

Summarizing, the number of polynomials satisfying the condition of the lemma is at most the coefficient of
$u^{n/b}$ in \[ \frac{1}{1-uq} \left( \frac{1}{1-uq^b} \right)^{1/b} .\] This is equal to
\[ \sum_{r=0}^{n/b} q^{n/b-r} q^{rb} \cdot {\rm Coefficient \ of \ } u^r \ {\rm in} \ (1-u)^{-1/b} .\] From Lemma 3.2 of \cite{FG2}, for $r \geq 1$
the coefficient of $u^r$ in $(1-u)^{-1/b}$ is at most $\frac{A}{b \sqrt{r}}$ for a universal constant $A$.
Dividing the sum into two summands (one for $0 \leq r \leq n/(2b)$ and one for $n/(2b) \leq r \leq n/b$),
the lemma easily follows.
\end{proof}

We obtain the following result.

\begin{theorem} \label{GLmain}
\begin{enumerate}

\item For $b$ prime, the proportion of elements of $GL(n,q)$ contained in a conjugate of $GL(n/b,q^b).b$ is at most $\frac{A [1+\log_q(n+1)]^6 }{n^{1/2}}$ for a universal constant $A$.

\item The proportion of elements of $GL(n,q)$ contained in a conjugate of $GL(n/b,q^b).b$ for some prime $b$ is at most $\frac{A [1+\log_q(n+1)]^6 \log_2(n)}{n^{1/2}}$ for a universal constant $A$.

\end{enumerate}
\end{theorem}

\begin{proof} The second part of the theorem follows from the first part together with the fact that an integer $n$ has at most $\log_2(n)$ prime divisors; hence we
prove part one.

The proportion of elements of $GL(n,q)$ conjugate to an element of the group $GL(n/b,q^b).b$ outside of
$GL(n/b,q^b)$ is at most the number of classes outside of $GL(n/b,q^b)$ divided by the minimum centralizer size of a class.
As explained in Section 5 of \cite{FG2}, this is at most

 \[ \frac{C q^{n/2}(1+\log_q(n))}{q^n} \leq A/n^{1/2} ,\] where $C$ and $A$ are universal constants.

To upper bound the proportion of elements of $GL(n,q)$ conjugate to an element of $GL(n/b,q^b)$, note by Lemma \ref{boundpolyGL} that the
number of possible characteristic polynomials of such an element is at most $B q^n/n^{1/2}$ for a universal constant $B$. Applying Theorem
\ref{GLbound} proves the result.
\end{proof}

\subsection{$GU(n,q)$} \label{extU}

Next we treat the case of $GU(n/b,q^b).b$ inside of $GU(n,q)$ with $b$ an odd prime 
(there is no embedding if $b=2$).  

Lemma \ref{Ucond} gives conditions for an element of $GU(n,q)$ to lie in a conjugate of $GU(n/b,q^b)$.

\begin{lemma} \label{Ucond} Let $b$ be an odd prime. If an element of $GU(n,q)$ lies in a conjugate of $GU(n/b,q^b)$, then
every irreducible factor of its characteristic polynomial either has degree a multiple of $b$ or
has every Jordan block size occur with multiplicity a multiple of $b$.
\end{lemma}

\begin{proof}
Since $GU(n/b,q^b)$ embeds in $GL(2n/b,q^{2b})$, the result follows from the $GL$ case. 
\end{proof}

Lemma \ref{boundpolyU} bounds the number of characteristic polynomials of $GU(n,q)$ satisfying the conditions of
Lemma \ref{Ucond}. It is part 3 of Proposition 2.7 of \cite{EG}.

\begin{lemma} \label{boundpolyU} Let $b$ be an odd prime. The number of characteristic polynomials of $GU(n,q)$ such that every
irreducible factor either has degree a multiple of $b$ or occurs with multiplicity a multiple of $b$ is at
most $B q^n/n^{2/3}$ for a universal constant $B$.
\end{lemma}

We obtain the following result.

\begin{theorem} \label{GUmain}
\begin{enumerate}

\item For $b$ an odd prime, the proportion of elements of $GU(n,q)$ in a conjugate of $GU(n/b,q^b).b$ is at most $\frac{A [2+\log_q(n+1)]}{n^{2/3}}$ for a universal constant $A$.

\item The proportion of elements of $GU(n,q)$ contained in a conjugate of $GU(n/b,q^b).b$ for some odd prime $b$ is at most $\frac{A [2+\log_q(n+1)] \log_2(n)}{n^{2/3}}$ for a universal constant $A$.

\end{enumerate}
\end{theorem}

\begin{proof} The second part of the theorem follows from the first part together with the fact that an integer $n$ has at most $\log_2(n)$ prime divisors; hence we
prove part one.

The proportion of elements of $GU(n,q)$ conjugate to an element of the group $GU(n/b,q^b).b$ outside of $GU(n/b,q^b)$ is at most the number of
$GU(n,q)$ classes
outside of $GU(n/b,q^b)$ divided by the minimum centralizer size of a class. By Theorem 5.6 of \cite{FG2}, Theorem 1.1 of \cite{FG1}, and
Theorem 1.4 of \cite{FG1}, this is easily at most $A/n^{2/3}$.

To upper bound the proportion of elements of $GU(n,q)$ conjugate to an element of $GU(n/b,q^b)$, note by Lemma \ref{boundpolyU} that the
number of possible characteristic polynomials of such an element is at most $B q^n/n^{2/3}$ for a universal constant $B$. Applying Theorem
\ref{ubound} proves the result.
\end{proof}

 \subsection{$Sp(2n,q)$} \label{extSp}

This subsection treats the case of $Sp(2n/b,q^b).b$ inside of $Sp(2n,q)$.

To begin we consider the case that $b$ is odd.

Lemma \ref{Spcondodd} gives conditions for an element of $Sp(2n,q)$ to lie in a conjugate of $Sp(2n/b,q^b)$.   This follows immediately from the corresponding result for $GL$ 
mentioned above. 

\begin{lemma} \label{Spcondodd} Let $b$ be an odd prime dividing $n$. If an element of $Sp(2n,q)$ lies
in a conjugate of $Sp(2n/b,q^b)$, then every irreducible factor of its characteristic polynomial either
has degree a multiple of $b$ or has every Jordan block size occur with multiplicity a multiple of $b$.
\end{lemma}

The next lemma (part 2 of Proposition 2.7 of \cite{EG}) bounds the number of characteristic polynomials of $Sp(2n,q)$ satisfying the conditions of
Lemma \ref{Spcondodd}.

\begin{lemma} \label{boundpolySp} Let $b$ be an odd prime dividing $n$. The number of characteristic polynomials
of $Sp(2n,q)$ such that every irreducible factor either has degree a multiple of $b$ or occurs with multiplicity
a multiple of $b$ is at most $B q^n/n^{2/3}$ for a universal constant $B$.
\end{lemma}

The above ingredients lead to the following theorem.

\begin{theorem} Let $b$ be an odd prime dividing $n$. Then the proportion of elements of $Sp(2n,q)$ which are contained
in a conjugate of $Sp(2n/b,q^b).b$ is at most \[ \frac{A [\log_q(4n)+4]^3}{n^{2/3}}, \] where $A$ is a universal constant.
\end{theorem}

\begin{proof} The proportion of elements of $Sp(2n,q)$ which are contained in a conjugate of $Sp(2n/b,q^b).b$ outside
of $Sp(2n/b,q^b)$ is studied in \cite{FG2} using Shintani descent and tends to zero much faster than the needed rate.
To study the proportion of elements of $Sp(2n,q)$ contained in a conjugate of $Sp(2n/b,q^b)$, note by Lemma
\ref{boundpolySp} that the number of possible characteristic polynomials of such elements is at most $B q^n/n^{2/3}$ for
a universal constant $B$. Now apply Theorem \ref{sym}.
\end{proof}

Next we study the case of $b=2$.

Lemma \ref{Spcondeven} gives conditions for an element of $Sp(2n,q)$ to lie in a conjugate of $Sp(n,q^2)$.
This result follows by the corresponding result for $GL$ stated at the start of Section \ref{extGL}.

\begin{lemma} \label{Spcondeven} If an element of $Sp(2n,q)$ lies in a conjugate of $Sp(n,q^2)$, then 
every irreducible factor of the characteristic polynomial has even degree or occurs with even multiplicity. 
\end{lemma} 

Lemma \ref{boundpolySpeven} (part 2 of Proposition 2.7 of \cite{EG}) bounds the number of characteristic polynomials of $Sp(2n,q)$ satisfying the
condition of Lemma \ref{Spcondeven}.

\begin{lemma} \label{boundpolySpeven} The number of characteristic polynomials of $Sp(2n,q)$ satisfying the
condition of Lemma \ref{Spcondeven} is at most $C q^n/n^{1/4}$ for a universal constant $C$.
\end{lemma}

This leads to the following theorem.

\begin{theorem} The proportion of elements of $Sp(2n,q)$ which are contained in a conjugate of $Sp(n,q^2).2$
is at most \[ \frac{A [\log_q(4n)+4]^3}{n^{1/4}}, \] where $A$ is a universal constant.
\end{theorem}

\begin{proof} The proportion of elements of $Sp(2n,q)$ which are contained in a conjugate of $Sp(n,q^2).2$ outside
of $Sp(n,q^2)$ is studied in \cite{FG2} using Shintani descent and tends to zero much faster than the needed rate.
To study the proportion of elements of $Sp(2n,q)$ contained in a conjugate of $Sp(n,q^2)$, note by Lemma
\ref{boundpolySpeven} that the number of possible characteristic polynomials of such elements is at most $C q^n/n^{1/4}$ for
a universal constant $C$. Now apply Theorem \ref{sym}.
\end{proof}

\subsection{$O^{\pm}(n,q)$} \label{extO}

In this section we discuss the results for extension field subgroups in the orthogonal groups. The argument and bounds are
nearly identical to the symplectic case (we use the orthogonal bounds from Section \ref{subO} instead of the symplectic
bounds from Section \ref{subSp}).  

If we consider $SO^{\pm}(n,q)$, the proofs are identical.  
If $q$ is bounded, the estimate
for the number of elements with a given characteristic polynomial in $O^{\pm}(n,q)$  just changes  by a fixed constant
and so the argument applies.   

Thus, it suffices to consider the outer coset (i.e.  $O^{\epsilon}(n,q) \setminus{SO^{\epsilon}(n,q)}$) with $q$ increasing. 

If $q$ is odd and increasing, for any element $g$ in the outer coset, the characteristic polynomial 
is divisible by $(z-1)^a(z+1)^b$ with $ab$ odd.   It is straightforward to see the proportion of characteristic polynomials with $ab > 1$
(and so $ab \ge 3$)
is at most $C/q^2$ for a universal constant $C$ and so even accounting for 
the larger number of elements with a given characteristic polynomial, the probability 
that an element in that coset has a characteristic polynomial with $ab > 1$ goes to $0$ as $q \rightarrow \infty$.  If $a$ or $b=1$, 
then $g$ is not in an field extension subgroup.    This argument also applies to the case of a unitary subgroup of the orthogonal group.

If $q$ is even and $g$ is in the outer coset,  then $(z-1)^2$ divides the characteristic polynomial of $g$ and  the only
case we need to consider is the case where two is the precise power of $z-1$ occurring
and indeed we can assume that the trivial eigenspace of $g$ has dimension $1$ and this implies that $g$ is not in an extension
field subgroup (or a unitary group).

In \cite{FG2}, the imprimitive case was only dealt with for $SO^{\pm}(n,q)$.   One can extend this to $O^{\pm}(n,q)$.

\subsection{Other cases} \label{extother}

To begin we treat the case of $GU(n,q).2$ in $Sp(2n,q)$.

Lemma \ref{USPcond} gives conditions for an element of $Sp(2n,q)$ to lie in a conjugate of $GU(n,q)$.
Note that the condition is the same as for Lemma \ref{Spcondeven}. The proof follows by the 
corresponding result for $GL$.

\begin{lemma} \label{USPcond} If an element of $Sp(2n,q)$ lies in a conjugate of $GU(n,q)$, then for every
irreducible factor of its  characteristic polynomial, either the degree is even or it has even multiplicity.
\end{lemma} 
 
This leads to the following theorem.

\begin{theorem} The proportion of elements of $Sp(2n,q)$ which are contained in a conjugate of $GU(n,q).2$
is at most \[ \frac{A [\log_q(4n)+4]^3}{n^{1/4}}, \] where $A$ is a universal constant.
\end{theorem}

\begin{proof} The proportion of elements of $Sp(2n,q)$ which are contained in a conjugate of $GU(n,q).2$ outside
of $GU(n,q)$ is studied in \cite{FG2} and tends to zero much faster than the needed rate.
To study the proportion of elements of $Sp(2n,q)$ contained in a conjugate of $GU(n,q)$, note by Lemmas
\ref{boundpolySpeven} and \ref{USPcond} that the number of possible characteristic polynomials of such elements is at most $C q^n/n^{1/4}$ for
a universal constant $C$. Now apply Theorem \ref{sym}.
\end{proof}

An identical argument proves the following theorem. Note that there is an extra factor of $q$ in the numerator, since
the orthogonal version of Theorem \ref{sym} (given in Subsection \ref{subO}) has an extra $q$ in the numerator.   This can be dealt
with as in the previous subsection.  

\begin{theorem} \label{really} The proportion of elements of $O^{\pm}(2n,q)$ which are contained in a conjugate of $U(n,q).2$
is at most \[ \frac{A q [\log_q(4n)+4]^3}{n^{1/4}}, \] where $A$ is a universal constant.
\end{theorem}

\section{Bounded Rank and Lower Bounds} 

Instead of considering characteristic polynomials, we can count the number of elements with a given semisimple part (up to conjugacy).  
For the linear, unitary and symplectic groups, this is the same as having the same characteristic polynomial.  In the orthgonal groups,  these conditions
are not quite equivalent.

Let $G$ be a simply connected simple algebraic group over an algebraically closed field $k$.   
Recall that if $g \in G$,
then $g=su = us $ with $s$ semisimple and $u$ unipotent with $s$ and $u$ unique.
Let $s \in G$ be a fixed semisimple element.   Let $G_s$ be the set of all elements in $G$ whose semisimple
part is conjugate to $s$.  In the case of $SL$ and $Sp$,  this is precisely saying the elements have characteristic 
polynomial (on the natural module) equal to the characteristic polynomial of  $s$.   More generally, this is equivalent to 
saying that the characteristic  polynomials
are the same as for $s$ on any (rational) finite dimensional module.  
One has to modify this for the orthogonal group case (since it is not simply connected). 

We note the following elementary result.

\begin{theorem} \label{irred}  Let $G$ be a simply connected simple algebraic group over an algebraically closed field.
Fix $s \in G$ semisimple. 
\begin{enumerate} 
\item  There is a bijective morphism from $G/C_G(s) \times U_s$ to $G_s$ where $U_s$ is the variety of unipotent elements in $C_G(s)$.
\item $G_s$ is an irreducible subvariety of $G$ of codimension equal to the rank of $G$.
\item  The set of regular elements in $G_s$ is an open dense subvariety of $G_s$.  
\end{enumerate}
\end{theorem}

\begin{proof}    Consider the map on $G/C_G(s) \times U_s$ sending $gC_G(s), u$  to the element $g^{-1} su g$. 
Clearly $h \in G$ is in $G_s$
if and only if $h$ is conjugate to $su$ with $u$ unipotent in $C_G(s)$.  Thus, this morphism is bijective.      Since $C_G(s)$ is a connected
reductive group (since $G$ is simply connected) of the same rank as $G$, it follows that $U_s$ is the closure of the conjugacy class of regular unipotent
elements in $C_G(s)$ and so has codimension  in $C_G(s)$ equal to the rank of $G$ and so the result follows.

If $u$ is a regular unipotent element in $C_G(s)$, then $su$ is a regular element.  Since the set of regular unipotent elements of  $C_G(s)$ is dense
and open in the set of unipotent elements of $C_G(s)$ the last result follows.  
\end{proof}

It is well known that up to conjugacy there are only finitely many conjugacy classes of centralizers of semisimple elements (and almost all them
have positive dimensional centers) and so one can use this (for fixed $G$) to get exact formulas for the size of $G_s(q)$.  Our main theorem gives an
explicit upper bound.  
This also shows that for $G$ of fixed rank $r$, $|G_s(q)| = q^{d-r} + O(q^{d-r-1})$ where $d = \dim G$.  

We now compute the (exact) lower bound in the case of $G=SL(n,q)$.  Since characteristic polynomials are invariant under conjugation in $GL(n,q)$,
we work in the bigger group.   Since every characteristic polynomial occurs as the characteristic polynomial of a regular element 
(i.e. its centralizer has dimension $n$ or equivalently it is similar to its companion matrix), it suffices to give a lower bound on the size of
the $GL$ class of a regular element. If $x \in GL(n,q)$, then the centralizer of $x$ has size less than $q^n$  (since it is the group of invertible elements
in the algebra generated by $x$ in matrices).  Thus, a lower bound for the number of elements with a given characteristic polynomial is 

$$
|GL(n,q)|/(q^n-1) = |GL(n-1,q)|q^{n-1}.
$$

This  lower bound is also immediate from the formula \eqref{starred}.  One can obtain similar lower bounds for the other cases. 

\section{Maximal Tori and Invariable generation}

We let $H$ be a group and for $x,y \in H$, let $y^x = x^{-1}yx$.   Let $x^H$ denote the conjugacy class of $x$ in $H$. 
Recall that we say $g_1, \ldots, g_r$ invariably generate $G$ if and only if $G= \langle g_1^{x_1}, \ldots, g_r^{x_r} \rangle$ for all
$x_i \in G$.   Similarly, we say   $g_1, \ldots, g_r \in  H \subseteq GL(n,q)$ is invariably  irreducible in $H$  if $\langle g_1^{x_1}, \ldots, g_r^{x_r} \rangle$ is irreducible for all $x_i \in H$.   We  similarly define invariably transitive for subgroups of the symmetric group.  

We first recall a result of Steinberg (see \cite[3.4.1]{Ca}).  

\begin{theorem} \label{maximal tori}   Let $G$ be a simply connected simple algebraic group over an algebraically
closed field $k$ in positive characteristic $p$.  Let $F$ be a Lang-Steinberg endomorphism of $G$ (so $G^F$ is finite). 
Let $s \in G^F$ be semisimple.  The number of $F$-stable maximal tori containing $s$ is the number of unipotent elements  in $C_G(s)^F$.
\end{theorem}

\begin{proof}  Steinberg proves this for $s=1$ (for any reductive group).   If $G$ is simply connected, then 
$C_G(s)$ is 
a connected reductive group and clearly any $F$-stable maximal torus containing $s$ is an $F$-stable maximal torus
of $C_G(s)$.
\end{proof}

Let $\mathcal{T}$ denote the set of $F$-stable maximal tori of $G$ and set $\Omega =\{T^F | T \in \mathcal{T} \}$. 
We view the elements of $\Omega$ as disjoint sets (i.e. we can rename them with sets of the same size but disjoint).

The previous result allows one to define a bijection $\phi$ between $G^F$ and the (disjoint) union of $T^F$  (by choosing
any bijection between elements in $G^F$ with fixed semisimple part $s$ and the $F$-stable maximal
tori containing $s$).

We now restrict our attention to  $SL(n,q)$.  The $G^F$ conjugacy classes of $F$-stable maximal tori are in bijection with
the conjugacy classes of   $W=S_n$, the Weyl group of $G$.   Let $g = su=us  \in SL(n,q)$ with $s$ semisimple and $u$ unipotent.
Note all such elements have the same characteristic polynomial $f_s(x)$.   Also note that $g$ leaves invariant some $e$-dimensional subspace
if and only if $s$ does  if and only if $f_s$ has an $e$-dimensional factor  (not necessarily irreducible).   Note that if $s$ is in a conjugate
of  $T_w$ and $s$ leaves invariant no $e$-dimensional subspace, then $w$ leaves invariant no subset of size $e$.   

This shows that if $X_e$ is  the set of elements of $SL(n,q)$ which do not leave invariant any subspace of dimension $e$, then
$|X_e| < \sum |T^F|$ where the sum is over all $T \in \mathcal{T}$  conjugate to some $T_w$ with $w$ leaving invariant no subset of size $e$.
In particular, we can now show:

\begin{theorem}  Let $e$ be a positive integer less than $n$.  The probability that a random element of 
$SL(n,q)$ (or $GL(n,q)$) leaves no $e$-dimensional subspace invariant  is
less than the probability that a random element of $S_n$ leaves no subset of size $e$ invariant. 
\end{theorem}  

\begin{proof} Let $n_w$ denote the number of $G^F$ conjugates of $T_w$. Then $n_w$ is equal to
\[ \frac{|G^F|}{|N_{G^F}(T_w)|} = \frac{|G^F|/|C_{G^F}(T_w)|} {|N_{G^F}(T_w)|/|C_{G^F}(T_w)|},\]
where $N$ denotes normalizer and $C$ denotes centralizer. Now the numerator is equal to $|G^F|/|T_w^F|$ and
the denominator is equal to $|C_W(w)|$. Thus \[ \frac{1}{|G^F|} n_w |T_w^F| = \frac{1}{|C_W(w)|} = \frac{|w^W|}{|W|}.\]
(See \cite{FG1.5} where this argument was used extensively).

Summing over all $w$ that do not leave invariant any subset of size $e$
 shows that the probability that a random element of $SL(n,q)$ leaves invariant some $e$-dimensional subspace is less than
the probability that a random element of $W$ leaves invariant a subset of size $e$. The inequality is strict since any maximal torus contains the identity.
\end{proof} 

Similarly, we can replace $e$ by any collection of integers between $1$ and $n-1$ (in particular the entire set). 

The same argument applies to $r$-tuples of elements in $G^F$ and $r$-tuples of $F$-stable maximal tori for any positive integer $r$.
   In particular, this gives:

\begin{theorem}  The probability that $r$ random elements of $SL(n,q)$ invariably generate an irreducible subgroup is less than the probability that 
$r$ random elements of $W$ invariably generate a transitive subgroup, 
\end{theorem}

If we take $r=3$, then it follows by \cite{EFG, PPR} that the limit as $n \rightarrow \infty$ of the probability that $3$ random elements of $W$ invariably  generate  a transitive subgroup
goes to $0$ and so we obtain:

\begin{cor} \label{McK}  The probability that $3$ random elements of $SL(n,q)$ invariably generate an irreducible subgroup 
goes to $0$ as $n$ goes to $\infty$ (independently of $q$).  
\end{cor}  

We note that McKemmie \cite{Mc} proved the previous result under the assumption that both $n$ and $q$ approach $\infty$.   She also proved that there exists a fixed $q_0$ so that if $q > q_0$, then the probability that $4$ random elements of a finite simple group of  Lie type of rank $r$  over the field of
$q$ elements invariably generate an irreducible subgroup is bounded away from $0$.  
   Unfortunately, the inequality above goes the wrong way and so the argument above does not prove (for $r=4$) that McKemmie's result holds for all $q$.

This proof does not immediately extend to the other classical groups.  The main issues are that there are different types of subspaces of each dimension
and the connection between the dimensions of nondegenerate subspaces fixed by $g$ and its semisimple part is complicated.

\end{document}